\documentclass{amsart}

\usepackage{amssymb,amsmath,amsthm}
\usepackage{color}
\makeindex
\usepackage{enumerate}
\usepackage{dsfont}

\newtheorem{thm}{Theorem}
\newtheorem{theorem}[thm]{Theorem}

\newtheorem{lemma}[thm]{Lemma}
\newtheorem{proposition}[thm]{Proposition}

\begin{document}

\title[]{Logarithmic Sobolev Inequalities for Mollified Compactly Supported Measures}

\author[]{David Zimmermann}
\address{Department of Mathematics\\
  University of California \\
  San Diego 92093}
\email{dszimmer@math.ucsd.edu}
\maketitle

\begin{abstract}
We show that the convolution of a compactly supported measure on $\mathbb{R}$ with a Gaussian measure satisfies a
logarithmic Sobolev inequality (LSI). We use this result to give a new proof of a classical result in random matrix theory that states that, under
certain hypotheses, the empirical law of eigenvalues of a sequence of random real symmetric matrices converges weakly in probability
to its mean. We then examine the optimal constants in the LSIs for the convolved measures in terms of the variance of the convolving Gaussian.
We conclude with partial results on the extension of our main theorem to higher dimensions.
\end{abstract}

\section{Introduction}

A probability measure $\mu$ on $\mathbb{R}^n$ is said to satisfy a logarithmic Sobolev inequality (LSI) with constant $c\in \mathbb{R}$ if
\begin{equation}\label{eq:lsi}
\mathrm{Ent}_\mu(g^2)\leq 2c\int|\nabla g|^2 d\mu
\end{equation}
for all smooth and sufficiently integrable functions $g:\mathbb{R}^n\rightarrow\mathbb{R}$, where $\mathrm{Ent}_\mu$, called the entropy functional, is defined as
$$
\mathrm{Ent}_\mu (f)\equiv \int f \log f \mbox{ }d\mu- \int f \mbox{ }d\mu \log\int f \mbox{ }d\mu
$$
for measurable $f\geq 0$. The smallest $c$ for which (\ref{eq:lsi}) holds is called the log Sobolev constant for $\mu$. We will restrict our attention
to measures on $\mathbb{R}$ until the end of the paper, in Section \ref{sec:higherdim}.

LSIs show up as an important tool in many areas of mathematics, such as geometry \cite{Ba94, Ba97, BH97, Da87, Da90, DS84, Le96},
probability \cite{BT06, DS96, GZ03}, and optimal transport \cite{Le01, Vi03}, as well as statistical physics \cite{Ya96, Ya97, Ze92}.
A 2003 paper of Ledoux \cite{Le03} uses LSI (in its equivalent form {\it hypercontractivity}, see \cite{Gr75}) to determine tail bounds for the largest eigenvalue of
a large symmetric random matrix with Gaussian entries. Another important application of LSI in probability is the {\it Herbst inequality} (see \cite{GR98}, p.301, Ex. 3.4):

\begin{theorem}\label{thm:herbst}
$\mathrm{(Herbst).}$  Let $\mu$ be a probability measure on $\mathbb{R}^n$ satisfying a LSI with constant $c$, and let $F:\mathbb{R}^n\rightarrow\mathbb{R}$
be Lipschitz. Then for all $\lambda\in\mathbb{R}$,
$$
\mu \left\{\left|F-\int F\hspace{1mm}d\mu\right|\geq \lambda \right\} \leq 2\exp\left(-\frac{\lambda^2}{2c||F||_{\mathrm{Lip}}^2}\right).
$$
\end{theorem}

Because of the widespread utility of LSI, it is of great interest to know which measures satisfy a LSI, and for those that do, what
the optimal constants are. There are some known sufficient conditions on $\mu$ in order for $\mu$ to satisfy a LSI (for example, \cite{BE85, BL06, BL00, HS87}),
as well as some known necessary conditions (for example, Theorem \ref{thm:herbst} above implies that $\mu$ must have sub-Gaussian tails if it
satisfies a LSI). Surprisingly absent from the literature is the idea of approximation of arbitrary measures by measures that
satisfy a LSI; this will be the focus of this paper. We will approximate by using convolution with Gaussian measures.

Convolution of an arbitrary measure with a Gaussian does not necessarily yield a LSI; for example, consider the exponential distribution on $\mathbb{R}$:
$d\mu(t)=\exp(-t)\hspace{1mm}dt$, $t\geq 0$. The right tail is not sub-Gaussian; therefore $\mu$ does not satisfy a LSI. If we convolve $\mu$ with the standard Gaussian measure, then the right tail of the convolved measure has density $p$ given by
\begin{align*}
p(x)&=\int_{-\infty}^\infty\frac{1}{\sqrt{2\pi}}\exp\left(-\frac{y^2}{2}\right)\exp(-(x-y))\cdot\mathds{1}_{\{y>0\}}(x-y)dy\\
&=\exp\left(-x+\frac{1}{2}\right)\int_{-\infty}^x\frac{1}{\sqrt{2\pi}}\exp\left(-\frac{(y+1)^2}{2}\right)dy\\
&\geq\frac{1}{2}\exp\left(-x+\frac{1}{2}\right), \hspace{7mm}\mbox{for}\hspace{5mm}x\geq -1.
\end{align*}
Thus the convolved measure still has an exponential, hence not sub-Gaussian, right tail and therefore does not satisfy a LSI either. So this approximation
scheme does not work in general. However, if we restrict our attention to compactly supported measures, then convolution will yield a LSI:

\begin{theorem}\label{thm:z}
Let $\mu$ be a compactly supported probability measure on $\mathbb{R}$. Let $\gamma_\delta$ be the
Gaussian measure centered at $0$ with variance $\delta>0$, i.e., $d\gamma_\delta=\frac{1}{\sqrt{2\pi\delta}}\exp
(\frac{-t^2}{2\delta})dt$. Then $\mu*\gamma_\delta$ satisfies a LSI with constant $c$ for some $c=c(\delta)$.
\end{theorem}

The main tool we use to prove Theorem \ref{thm:z} is the following theorem due to Bobkov and G\"{o}tze (see \cite{BG99}, p.25, Thm. 5.3):

\begin{theorem}\label{thm:bg}
$\mathrm{(Bobkov, G\ddot{o}tze).}$ Let $\mu$ be a Borel probability measure on $\mathbb{R}$ with distribution function
$F(x)=\mu((-\infty,x])$. Let $p$ be the density of the absolutely continuous part of $\mu$
with respect to Lebesgue measure, and let $m$ be a median of $\mu$. Let
\begin{align*}
D_0&=\sup_{x<m}\left(F(x)\cdot\log\frac{1}{F(x)}\cdot\int_x^m\frac{1}{p(t)}dt\right),\\
D_1&=\sup_{x>m}\left((1-F(x))\cdot\log\frac{1}{1-F(x)}\cdot\int_m^x\frac{1}{p(t)}dt\right),
\end{align*}
defining $D_0$ and $D_1$ to be zero if $\mu((-\infty,m))=0$ or  $\mu((m,\infty))=0$, respectively.
Then the log Sobolev constant $c$ for $\mu$ satisfies $\frac{1}{150}(D_0+D_1)\leq c\leq 468(D_0+D_1)$.
\end{theorem}
In particular, $\mu$ satisfies a LSI if and only if $D_0+D_1<\infty$.\\

Theorem \ref{thm:z} enables us to give a new proof of the universality theorem in random matrix theory that the empirical law of
eigenvalues of an $n\times n$ real symmetric random matrix converges weakly to its mean in probability as $n\rightarrow\infty$;
this is detailed in Section \ref{sec:app}.

\section{Proof of Main Theorem}

The key idea to the proof of Theorem \ref{thm:z} using Theorem \ref{thm:bg} is the fact that we can describe the tail behavior of the
convolution of a compactly supported measure with a Gaussian.\\

\begin{proof}

Suppose $\mathrm{supp}(\mu)\subseteq [a,b]$. We will apply Theorem \ref{thm:bg} to the
probability measure $\mu*\gamma_\delta$. We will show $D_0$ and $D_1$, as defined in Theorem \ref{thm:bg},
are finite; at the moment we consider $D_0$. Since $\gamma_\delta$ has a smooth density, $\mu*\gamma_\delta$ has a smooth density $p$.
Note that $p$ is nonzero everywhere since $\gamma_\delta$ has strictly positive density. We therefore want to show
$$
D_0=\sup_{x<m}\left(\int_{-\infty}^xp(t)dt\cdot\log\frac{1}{\int_{-\infty}^xp(t)dt}\cdot\int_x^m\frac{1}{p(t)}dt\right)
$$
is finite. Since the above expression is continuous in $x$ for all $x\in\mathbb{R}$, it is bounded on
every compact interval. We therefore only need to show that
$$
\limsup_{x\rightarrow-\infty}
\left(\int_{-\infty}^xp(t)dt\cdot\log\frac{1}{\int_{-\infty}^xp(t)dt}\cdot\int_x^m\frac{1}{p(t)}dt\right)
$$
is finite. We will do this by giving asymptotics for $\int_{-\infty}^xp(t)dt$ and $\int_x^m\frac{1}{p(t)}dt$.

\begin{lemma}\label{lem:densityasymptotic}
$$
\lim_{x\rightarrow-\infty}\frac{\delta p'(x)}{-x p(x)}=1.
$$
\end{lemma}

\begin{proof}
By definition of $p$,
$$
p(x)=\int \frac{1}{\sqrt{2\pi\delta}}\exp\left(\frac{-(x-t)^2}{2\delta}\right)d\mu(t),
$$
so
$$
\frac{p(x+h)-p(x)}{h}
=\int \frac{1}{\sqrt{2\pi\delta}}\frac{1}{h}
\left(\exp\left(\frac{-(x+h-t)^2}{2\delta}\right)-\exp\left(\frac{-(x-t)^2}{2\delta}\right)\right)d\mu(t).
$$
Since, by the Mean Value Theorem, the integrand in the above equation is dominated uniformly in $h$ by
$\max_{t\in\mathbb{R}}\frac{-t}{\delta\sqrt{2\pi\delta}}\exp(\frac{-t^2}{2\delta})<\infty$,
we can let $h\rightarrow 0$ and apply the Dominated Convergence Theorem to differentiate under the
integral and get
$$
p'(x)=\int \frac{-1}{\delta\sqrt{2\pi\delta}}(x-t)\exp\left(\frac{-(x-t)^2}{2\delta}\right)d\mu(t).
$$
Then
\begin{align*}
\frac{\delta p'(x)}{-x p(x)} & =
\frac{\delta\int \frac{-1}{\delta\sqrt{2\pi\delta}}(x-t)\exp\left(\frac{-(x-t)^2}{2\delta}\right)d\mu(t)}
{-x\int \frac{1}{\sqrt{2\pi\delta}}\exp\left(\frac{-(x-t)^2}{2\delta}\right)d\mu(t)}\\
&= \frac{\int(x-t)\exp\left(\frac{-(x-t)^2}{2\delta}\right)d\mu(t)}
{x\int \exp\left(\frac{-(x-t)^2}{2\delta}\right)d\mu(t)}\\
&= 1-\frac{\int t\exp\left(\frac{-(x-t)^2}{2\delta}\right)d\mu(t)}
{x\int \exp\left(\frac{-(x-t)^2}{2\delta}\right)d\mu(t)}.
\end{align*}
But
\begin{align*}
\left|{\frac{\int t\exp\left(\frac{-(x-t)^2}{2\delta}\right)d\mu(t)}
{x\int \exp\left(\frac{-(x-t)^2}{2\delta}\right)d\mu(t)}}\right|
&\leq {\frac{\int |t|\exp\left(\frac{-(x-t)^2}{2\delta}\right)d\mu(t)}
{|x|\int \exp\left(\frac{-(x-t)^2}{2\delta}\right)d\mu(t)}}\\
&\leq {\frac{\max(|a|,|b|)\int \exp\left(\frac{-(x-t)^2}{2\delta}\right)d\mu(t)}
{|x|\int \exp\left(\frac{-(x-t)^2}{2\delta}\right)d\mu(t)}}\\
&=\frac{\max(|a|,|b|)}{|x|}\\
&\rightarrow 0 \hspace{3mm}\mbox{as} \hspace{3mm}x\rightarrow-\infty,
\end{align*}
so $\frac{\delta p'(x)}{-x p(x)}\rightarrow 1$ as $x\rightarrow -\infty$. \\
\end{proof}

The next two lemmas give asymptotics for $\int_{-\infty}^xp(t)dt$ and $\int_x^m\frac{1}{p(t)}dt$.
We will say $f(x)\sim g(x)$ if $\frac{f(x)}{g(x)}\rightarrow 1$ as $x\rightarrow -\infty$.

\begin{lemma}\label{lem:asymptotic1}
$$
\int_{-\infty}^xp(t)dt \sim -\frac{\delta}{x}p(x).
$$
\end{lemma}

\begin{proof}
Observe that both $\int_{-\infty}^xp(t)dt$ and $-\frac{\delta}{x}p(x)$
tend to $0$ as $x\rightarrow -\infty$ and apply L'H\^{o}pital's Rule and Lemma \ref{lem:densityasymptotic}:
\begin{align*}
\lim_{x\rightarrow -\infty}\frac{\int_{-\infty}^xp(t)dt}{-\frac{\delta}{x}p(x)}
&=\lim_{x\rightarrow -\infty}\frac{p(x)}{\frac{\delta}{x^2}p(x)-\frac{\delta}{x}p'(x)}\\
&=\lim_{x\rightarrow -\infty}\frac{1}{\frac{\delta}{x^2}+\frac{\delta p'(x)}{-x p(x)}}\\
&=1
\end{align*}
\end{proof}

\begin{lemma}\label{lem:asymptotic2}
$$
\int_x^m\frac{1}{p(t)}dt \sim -\frac{\delta}{x p(x)}.
$$
\end{lemma}

Observe that this claim shows that the above integral asymptotically does not depend on $m$.
Since $p$ is continuous and nonzero on $\mathbb{R}$ and $m\in[a,b]$, $\int_x^m\frac{1}{p(t)}dt$
is finite for each $x$; and since $\frac{1}{p(t)}$ blows up as $x\rightarrow -\infty$, any dependence
on $m$ of $\int_x^m\frac{1}{p(t)}dt$ is diminished as $x\rightarrow -\infty$.

\begin{proof}
\begin{align*}
p(x)&=\int \frac{1}{\sqrt{2\pi\delta}}\exp\left(\frac{-(x-t)^2}{2\delta}\right)d\mu(t)\\
&\leq \int \frac{1}{\sqrt{2\pi\delta}}\exp\left(\frac{-(x-a)^2}{2\delta}\right)d\mu(t)\\
&=\frac{1}{\sqrt{2\pi\delta}}\exp\left(\frac{-(x-a)^2}{2\delta}\right),
\end{align*}
so that
$$
\frac{1}{p(x)}\geq\sqrt{2\pi\delta}\cdot\exp\left(\frac{(x-a)^2}{2\delta}\right).
$$
So $-\frac{\delta}{x p(x)}$ tends to $+\infty$ as $x\rightarrow-\infty$
and we can again use L'H\^{o}pital's Rule and Lemma \ref{lem:densityasymptotic}:
\begin{align*}
\lim_{x\rightarrow -\infty}\frac{\int_x^m \frac{1}{p(t)}dt}{-\frac{\delta}{xp(x)}}
&=\lim_{x\rightarrow -\infty}\frac{-\frac{1}{p(x)}}{\frac{\delta}{(xp(x))^2}(p(x)+xp'(x))}\\
&=\lim_{x\rightarrow -\infty}\frac{1}{\frac{-\delta}{x^2}+\frac{\delta p'(x)}{-x p(x)}}\\
&=1.
\end{align*}
\end{proof}

Before we proceed, we need the following fact about asymptotics of logs: if $f(x)\sim g(x)$ and
$g(x)\rightarrow\infty$ as $x\rightarrow-\infty$, then $\log f(x)\sim\log g(x)$. This follows by
observing that
$$
\frac{\log f(x)}{\log g(x)}=1+\frac{\log\left(\frac{f(x)}{g(x)}\right)}{\log g(x)}
$$
and letting $x\rightarrow-\infty$.\\

\begin{proposition}\label{prop:D0D1}
$D_0$ and $D_1$ are finite.
\end{proposition}

\begin{proof}
We first consider $D_0$. By the observations made at the beginning of this section, it suffices to show that
$$
\limsup_{x\rightarrow-\infty}
\left(\int_{-\infty}^xp(t)dt\cdot\log\frac{1}{\int_{-\infty}^xp(t)dt}\cdot\int_x^m\frac{1}{p(t)}dt\right)
$$
is finite. By Lemmas \ref{lem:asymptotic1} and \ref{lem:asymptotic2},
\begin{align*}
&\limsup_{x\rightarrow-\infty}
\left(\int_{-\infty}^xp(t)dt\cdot\log\frac{1}{\int_{-\infty}^xp(t)dt}\cdot\int_x^m\frac{1}{p(t)}dt\right)\\
=&\limsup_{x\rightarrow-\infty}
-\frac{\delta}{x}p(x)\cdot\log\left(-\frac{x}{\delta}\frac{1}{p(x)}\right)\cdot\left(-\frac{\delta}{xp(x)}\right)\\
=&\limsup_{x\rightarrow-\infty}
\frac{\delta^2}{x^2}\left(\log\left(\frac{-x}{\delta}\right)-\log p(x)\right)\\
=&\limsup_{x\rightarrow-\infty}\frac{\delta^2}{x^2}\left(-\log p(x)\right).
\end{align*}
We just now need to show that $\limsup_{x\rightarrow-\infty}\frac{\delta^2}{x^2}\left(-\log p(x)\right)<\infty.$
But for $x\leq a$,
\begin{align*}
p(x)&=\int \frac{1}{\sqrt{2\pi\delta}}\exp\left(\frac{-(x-t)^2}{2\delta}\right)d\mu(t)\\
&\geq \int \frac{1}{\sqrt{2\pi\delta}}\exp\left(\frac{-(x-b)^2}{2\delta}\right)d\mu(t)\\
&=\frac{1}{\sqrt{2\pi\delta}}\exp\left(\frac{-(x-b)^2}{2\delta}\right)
\end{align*}
so that
\begin{align*}
&\limsup_{x\rightarrow-\infty}\frac{\delta^2}{x^2}\left(-\log p(x)\right)\\
\leq&\limsup_{x\rightarrow-\infty}
-\frac{\delta^2}{x^2}\log\left(\frac{1}{\sqrt{2\pi\delta}}\exp\left(\frac{-(x-b)^2}{2\delta}\right) \right)\\
=&\limsup_{x\rightarrow-\infty}
-\frac{\delta^2}{x^2}\left(\log\left(\frac{1}{\sqrt{2\pi\delta}}\right)+\frac{-(x-b)^2}{2\delta}\right)\\
=&\frac{\delta}{2}<\infty.
\end{align*}
Therefore $D_0<\infty$.\\

The proof that $D_1<\infty$ is practically identical, the relevant ingredients being the following:
\begin{align*}
&1-F(x)=\int_x^\infty p(t)dt,\\
&\lim_{x\rightarrow+\infty}\frac{\delta p'(x)}{-x p(x)}=1,\\
&\int_x^\infty p(t)dt \sim \frac{\delta}{x}p(x)  \hspace{3mm}\mbox{as} \hspace{3mm} x\rightarrow+\infty, \hspace{3mm} \mbox{and}\\
&\int_m^x\frac{1}{p(t)}dt \sim \frac{\delta}{x p(x)} \hspace{3mm} \mbox{as}  \hspace{3mm}x\rightarrow+\infty.
\end{align*}
Details are omitted.
\end{proof}

Theorem 2 now immediately follows from Proposition \ref{prop:D0D1}.

\end{proof}

\section{Application to Random Matrices}\label{sec:app}

We now give an application of our theorem to random matrix theory. For each natural number $n$, let $Y_n$ be an $n\times n$
random real symmetric matrix whose upper triangular entries are i.i.d. with distribution $\nu$, and let
$X_n = \frac{1}{\sqrt{n}}Y_n$. By a classical result in random matrix theory due to Wigner \cite{Wi55,Wi58}, if $\nu$
has finite second moment, then the empirical law of eigenvalues of $X_n$ converges weakly to its mean in probability.
That is: let $\lambda^n_1,\lambda^n_2,\ldots,\lambda^n_n$ be the (necessarily real) eigenvalues of $X_n$, and let
$$
\mu_{X_n} = \frac{1}{n}\sum_{k=1}^n \delta_{\lambda^n_k}.
$$
Then for all $\epsilon>0$ and all Lipschitz $f:\mathbb{R}\rightarrow\mathbb{R}$,
$$
\mathbb{P}\left(\left|\int f \mbox{ } d\mu_{X_n}-\mathbb{E}\left(\int f \mbox{ } d\mu_{X_n}\right)\right|\geq \epsilon\right)\rightarrow 0
$$
as $n\rightarrow\infty$.

The original proof of this fact was combinatorial in nature; in 2008, Guionnet (see \cite{Gu09}, p.70, Thm. 6.6) proved this convergence using
logarithmic Sobolev inequalities in the special case where $\nu$ satisfies a LSI, using the following theorem:

\begin{theorem}\label{thm:g}
$\mathrm{(Guionnet }$ \cite{Gu09}$\mathrm{).}$  Let $\nu, X_n$ be as above. If $\nu$ satisfies a LSI with constant $c$, then for all $\epsilon>0$ and all Lipschitz $f:\mathbb{R}\rightarrow\mathbb{R}$,
$$
\mathbb{P}\left(\left|\int f \mbox{ } d\mu_{X_n}-\mathbb{E}\left(\int f \mbox{ } d\mu_{X_n}\right)\right|\geq \epsilon\right)\leq 2\exp\left(\frac{-n^2 \epsilon^2}{4c||f||_{\mathrm{Lip}}^2}\right),
$$
where
$$
||f||_{\mathrm{Lip}}=\sup_{x\neq y}\frac{|f(x)-f(y)|}{|x-y|}.
$$
\end{theorem}

The convergence proven by Wigner is in fact almost sure convergence; if $\nu$ also satisfies a LSI, then one could also use Theorem \ref{thm:g} and the Borel-Cantelli Lemma to deduce almost sure convergence. Using Theorems \ref{thm:z} and \ref{thm:g}, we give a new proof of the slightly weaker result that the empirical law of
eigenvalues $\mu_{X_n}$ converges weakly to its mean in probability, {\em regardless} of whether or not $\nu$ satisfies a LSI (we still assume finite second moment).
We first state a lemma from matrix theory (see \cite{HW53}, p.37, Thm. 1, and p.39, Rk. 2):

\begin{lemma}\label{lem:hw}
$\mathrm{(Hoffman\mbox{, }Wielandt).}$ Let $A, B$ be symmetric $n\times n$ matrices with eigenvalues $\lambda^A_1\leq\lambda^A_2\leq\ldots\leq\lambda^A_n$ and $\lambda^B_1\leq\lambda^B_2\leq\ldots\leq\lambda^B_n$. Then
$$
\sum_{i=1}^n (\lambda^A_i-\lambda^B_i)^2\leq\mathrm{Tr}[(A-B)^2].
$$
\end{lemma}

We now prove weak convergence in probability of the empirical law eigenvalues of $X_n$ to its mean.

\begin{theorem}\label{thm:evalconverge}
Let $\nu$ be a probability measure on $\mathbb{R}$ with finite second moment.
For each natural number $n$, let $Y_n$ be an $n\times n$ random real symmetric matrix
whose upper triangular entries are i.i.d. with distribution $\nu$, and let $X_n=\frac{1}{\sqrt{n}}Y_n$.
Then the empirical law of eigenvalues $\mu_{X_n}$ of $X_n$ converges weakly to its mean in probability.
\end{theorem}

\begin{proof}
Suppose $\nu$ has finite second moment. We will first suppose that the entries of $Y_n$ are bounded, so that $\nu$ is compactly supported;
we will remove this assumption at the end of the proof (on page \pageref{generalcase}).

Let $\epsilon >0$, and let $f:\mathbb{R}\rightarrow\mathbb{R}$ be Lipschitz. For each $n$, let $\widetilde{Y}_n=Y_n+\sqrt{\delta}G_n$ and $\widetilde{X}_n=\frac{1}{\sqrt{n}}\widetilde{Y}_n$, where $G_n$ is a random symmetric matrix whose upper triangular entries are independent (and independent of $Y_n$) and all have a Gaussian distribution with mean 0 and variance 1, and $\delta=\delta(n)$ is a positive real number that we will send to 0 as $n\rightarrow\infty$ (later in the proof).

By construction, the distributions of the entries of $\widetilde{Y}_n$ are the convolution of $\nu$ (a compactly supported measure) with a Gaussian of variance $\delta$. So by Theorem \ref{thm:z}, the distributions of the entries of $\widetilde{Y}_n$ satisfy a LSI with constant $c$ for some $c=c(\delta)$. Now
\begin{equation}\label{eqn:epsover3}
\begin{aligned}
\mathbb{P}\left(\left|\int f \mbox{ } d\mu_{X_n}-\mathbb{E}\left(\int f \mbox{ } d\mu_{X_n}\right)\right|\geq \epsilon\right)
&\leq\mathbb{P}\left(\left|\int f \mbox{ } d\mu_{X_n}-\int f \mbox{ } d\mu_{\widetilde{X}_n}\right|\geq \frac{\epsilon}{3}\right)\\
&+\mathbb{P}\left(\left|\int f \mbox{ } d\mu_{\widetilde{X}_n}-\mathbb{E}\left(\int f \mbox{ } d\mu_{\widetilde{X}_n}\right)\right|\geq \frac{\epsilon}{3}\right)\\
&+\mathbb{P}\left(\left|\mathbb{E}\left(\int f \mbox{ } d\mu_{\widetilde{X}_n}\right)-\mathbb{E}\left(\int f \mbox{ } d\mu_{X_n}\right)\right|\geq \frac{\epsilon}{3}\right),
\end{aligned}
\end{equation}
where $\mu_{X_n}$ and $\mu_{\widetilde{X}_n}$ are the empirical laws of eigenvalues for $X_n$ and $\widetilde{X}_n$. We now show that each of the three terms on the right-hand side of (\ref{eqn:epsover3}) tend to 0 as $n\rightarrow\infty$.

\begin{lemma}\label{lem:epsover31}
$$
\mathbb{P}\left(\left|\int f \mbox{ } d\mu_{X_n}-\int f \mbox{ } d\mu_{\widetilde{X}_n}\right|\geq \frac{\epsilon}{3}\right)
\leq\frac{9 ||f||_{\mathrm{Lip}}^2}{\epsilon^2}\mbox{ }\delta.
$$
\end{lemma}

\begin{proof}
Let $\lambda^n_1\leq\lambda^n_2\leq\ldots\leq\lambda^n_n$ and $\widetilde{\lambda}^n_1\leq\widetilde{\lambda}^n_2\leq\ldots\leq\widetilde{\lambda}^n_n$ be the eigenvalues of $X_n$ and $\widetilde{X}_n$. Then by the Cauchy-Schwarz inequality and Lemma \ref{lem:hw},
\begin{align*}
\left|\int f \mbox{ } d\mu_{X_n}-\int f \mbox{ } d\mu_{\widetilde{X}_n}\right|
&= \left|\frac{1}{n}\sum_{i=1}^n f(\lambda^n_i)-f(\widetilde{\lambda}^n_i)\right|\\
&\leq \frac{1}{n}\sum_{i=1}^n ||f||_{\mathrm{Lip}}\left|\lambda^n_i-\widetilde{\lambda}^n_i\right|\\
&\leq \frac{||f||_{\mathrm{Lip}}}{\sqrt{n}}\left(\sum_{i=1}^n (\lambda^n_i-\widetilde{\lambda}^n_i)^2\right)^{1/2}\\
&\leq \frac{||f||_{\mathrm{Lip}}}{\sqrt{n}}\left(\mathrm{Tr}[(X_n-\widetilde{X}_n)^2]\right)^{1/2}.
\end{align*}
By Markov's inequality, we therefore have
\begin{align*}
\mathbb{P}\left(\left|\int f \mbox{ } d\mu_{X_n}-\int f \mbox{ } d\mu_{\widetilde{X}_n}\right|\geq \frac{\epsilon}{3}\right)
&\leq \mathbb{P}\left(\frac{||f||_{\mathrm{Lip}}}{\sqrt{n}}\left(\mathrm{Tr}[(X_n-\widetilde{X}_n)^2]\right)^{1/2}\geq \frac{\epsilon}{3}\right)\\
&=\mathbb{P}\left(\mathrm{Tr}[(X_n-\widetilde{X}_n)^2]\geq \frac{\epsilon^2 n}{9 ||f||_{\mathrm{Lip}}^2}\right)\\
&\leq \frac{9 ||f||_{\mathrm{Lip}}^2}{\epsilon^2 n}\mbox{ }\mathbb{E}\left(\mathrm{Tr}[(X_n-\widetilde{X}_n)^2]\right)\\
&=\frac{9 ||f||_{\mathrm{Lip}}^2}{\epsilon^2 n}\sum_{1\leq i,j \leq n}\mathbb{E}\left(([X_n]_{ij}-[\widetilde{X}_n]_{ij})^2\right)\\
&=\frac{9 ||f||_{\mathrm{Lip}}^2}{\epsilon^2 n}\sum_{1\leq i,j \leq n}\mathbb{E}\left(\frac{\delta}{n}[G_n]_{ij}^2\right)\\
&=\frac{9 ||f||_{\mathrm{Lip}}^2}{\epsilon^2 n}\cdot\delta n,\\
&=\frac{9 ||f||_{\mathrm{Lip}}^2}{\epsilon^2}\mbox{ }\delta,
\end{align*}
the second to last equality being because the entries of $G_n$ have mean 0 and variance 1.
\end{proof}

\begin{lemma}\label{lem:epsover32}
$$
\mathbb{P}\left(\left|\int f \mbox{ } d\mu_{\widetilde{X}_n}-\mathbb{E}\left(\int f \mbox{ } d\mu_{\widetilde{X}_n}\right)\right|\geq \frac{\epsilon}{3}\right)
\leq 2\exp\left(\frac{-n^2 \epsilon^2}{36 c||f||_{\mathrm{Lip}}^2}\right),
$$
where $c=c(\delta)$ is the log Sobolev constant for $\nu*\gamma_\delta$.
\end{lemma}

\begin{proof}
This is immediate from Theorem \ref{thm:g}.
\end{proof}

\begin{lemma}\label{lem:epsover33}
If $\delta(n)\rightarrow0$ as $n\rightarrow\infty$, then
$$
\mathbb{P}\left(\left|\mathbb{E}\left(\int f \mbox{ } d\mu_{\widetilde{X}_n}\right)-\mathbb{E}\left(\int f \mbox{ } d\mu_{X_n}\right)\right|\geq \frac{\epsilon}{3}\right) = 0
$$
for all $n$ sufficiently large.
\end{lemma}

\begin{proof}
Note that the sequence
$$
\left|\mathbb{E}\left(\int f \mbox{ } d\mu_{\widetilde{X}_n}\right)-\mathbb{E}\left(\int f \mbox{ } d\mu_{X_n}\right)\right|
$$
is a sequence of real numbers, so the above probability will eventually be equal to 0 if
$\left|\mathbb{E}\left(\int f \mbox{ } d\mu_{\widetilde{X}_n}\right)-\mathbb{E}\left(\int f \mbox{ } d\mu_{X_n}\right)\right|$
converges to 0 as $n\rightarrow\infty$. Doing similar estimates as in Lemma \ref{lem:epsover31}, we get
\begin{align*}
\left|\mathbb{E}\left(\int f \mbox{ } d\mu_{\widetilde{X}_n}\right)-\mathbb{E}\left(\int f \mbox{ } d\mu_{X_n}\right)\right|
&=\left|\mathbb{E}\left(\int f \mbox{ } d\mu_{\widetilde{X}_n}-\int f \mbox{ } d\mu_{X_n}\right)\right|\\
&\leq\mathbb{E}\left(\left|\int f \mbox{ } d\mu_{\widetilde{X}_n}-\int f \mbox{ } d\mu_{X_n}\right|\right)\\
&\leq \mathbb{E}\left(\frac{||f||_{\mathrm{Lip}}}{\sqrt{n}}\left(\mathrm{Tr}[(X_n-\widetilde{X}_n)^2]\right)^{1/2}\right)\\
&\leq \frac{||f||_{\mathrm{Lip}}}{\sqrt{n}}\left(\mathbb{E}\left(\mathrm{Tr}[(X_n-\widetilde{X}_n)^2]\right)\right)^{1/2}\\
&= \frac{||f||_{\mathrm{Lip}}}{\sqrt{n}}\cdot(\delta\cdot n)^{1/2}\\
&=||f||_{\mathrm{Lip}}\sqrt{\delta(n)}\\
&\rightarrow 0 \hspace{3mm}\mbox{as} \hspace{3mm}n\rightarrow\infty,
\end{align*}
the last inequality following from H\"{o}lder's inequality applied to $\left(\mathrm{Tr}[(X_n-\widetilde{X}_n)^2]\right)^{1/2}$
and the constant function 1. So
$\mathbb{P}\left(\left|\mathbb{E}\left(\int f \mbox{ } d\mu_{\widetilde{X}_n}\right)-\mathbb{E}\left(\int f \mbox{ } d\mu_{X_n}\right)\right|\geq \frac{\epsilon}{3}\right)=0$ for all sufficiently large $n$.
\end{proof}

We now construct our $\delta=\delta(n)$ so that $\delta\rightarrow 0$ at the appropriate rate as $n\rightarrow\infty$.

\begin{lemma}\label{lem:deltato0}
There exists a sequence $\delta(n)$ of positive real numbers such that
\begin{align*}
&(i) \hspace{3mm}\delta(n)\rightarrow 0 \hspace{2mm}\mbox{as} \hspace{2mm}n\rightarrow\infty, \hspace{3mm}\mbox{and}\\
&(ii)\hspace{3mm}\mbox{For all sufficiently large }n, \mbox{ } c(\delta(n))\leq n.
\end{align*}
\end{lemma}

\begin{proof}
For each $n$,
let $a_n=\inf\{\delta>0\mbox{ }|\mbox{  }c(\delta)\leq n\}$. Clearly $a_n$ is finite for all sufficiently large $n$. Also, $a_n\rightarrow 0$ as $n\rightarrow \infty$ since for every $\delta>0$ there is some $n_0$ such that $c(\delta)\leq n$ so that $a_n\leq\delta$ for $n\geq n_0$.

Now for each $n$ sufficiently large, let $\delta(n)$ be any number in the interval $[a_n,a_n+1/n]$ such that $c(\delta(n))\leq n$ (such a number exists by definition of $a_n$). Then since $a_n\rightarrow 0$, we have $\delta(n)\rightarrow 0$ as $n\rightarrow\infty$.
\end{proof}

To conclude the proof of Theorem \ref{thm:evalconverge} (in the case where the entries of the $Y_n$ are bounded), we apply Lemmas \ref{lem:epsover31},\ref{lem:epsover32},\ref{lem:epsover33}, and
\ref{lem:deltato0} to (\ref{eqn:epsover3}) to get that for sufficiently large $n$,
\begin{align*}
\mathbb{P}\left(\left|\int f \mbox{ } d\mu_{X_n}-\mathbb{E}\left(\int f \mbox{ } d\mu_{X_n}\right)\right|\geq \epsilon\right)
&\leq \frac{9 ||f||_{\mathrm{Lip}}^2}{\epsilon^2}\mbox{ }\delta(n) + 2\exp\left(\frac{-n^2 \epsilon^2}{36c(\delta(n))||f||_{\mathrm{Lip}}^2}\right) + 0\\
&\leq \frac{9 ||f||_{\mathrm{Lip}}^2}{\epsilon^2}\mbox{ }\delta(n) + 2\exp\left(\frac{-n\epsilon^2}{36||f||_{\mathrm{Lip}}^2}\right)\\
&\rightarrow 0 \hspace{3mm}\mbox{as} \hspace{3mm}n\rightarrow\infty.
\end{align*}

We therefore have weak convergence in probability.\\

\label{generalcase}To obtain convergence in the general case where the entries of $Y_n$ need not be bounded, we apply a standard ``cutoff" argument.
Given $C\geq0$, let $\widehat{Y}_n$ be the matrix such that
$[\widehat{Y}_n]_{ij}=[Y_n]_{ij}\cdot\mathds{1}_{\{|[Y_n]_{ij}|<C\}}$, and let $\widehat{X}_n=\frac{1}{\sqrt{n}}\widehat{Y}_n$.
Note the entries of $\widehat{Y}_n$ are bounded, hence have compactly supported distribution. (We remark that it is not necessary to normalize $\widehat{Y}_n$
since no assumptions on the values of the mean or the variance of $Y_n$ were used.)

Let $\epsilon>0,$ and let $\eta>0$. Denote $[Y_n]_{1,1}$ by $Y_{11}$.
Since $\nu$ has finite second moment, there is some constant $C$ such that
$\mathbb{E}((Y_{11}-\widehat{Y}_{11})^2)<\min(1,\eta)\cdot\epsilon^2/(9||f||_{\mathrm{Lip}}^2)$.
Then, similarly as before, we have
\begin{equation}\label{eq:otherepsover3}
\begin{aligned}
\mathbb{P}\left(\left|\int f \mbox{ } d\mu_{X_n}-\mathbb{E}\left(\int f \mbox{ } d\mu_{X_n}\right)\right|\geq \epsilon\right)
&\leq\mathbb{P}\left(\left|\int f \mbox{ } d\mu_{X_n}-\int f \mbox{ } d\mu_{\widehat{X}_n}\right|\geq \frac{\epsilon}{3}\right)\\
&+\mathbb{P}\left(\left|\int f \mbox{ } d\mu_{\widehat{X}_n}-\mathbb{E}\left(\int f \mbox{ } d\mu_{\widehat{X}_n}\right)\right|\geq \frac{\epsilon}{3}\right)\\
&+\mathbb{P}\left(\left|\mathbb{E}\left(\int f \mbox{ } d\mu_{\widehat{X}_n}\right)-\mathbb{E}\left(\int f \mbox{ } d\mu_{X_n}\right)\right|\geq \frac{\epsilon}{3}\right).
\end{aligned}
\end{equation}

The first term on the right-hand side of (\ref{eq:otherepsover3}) is bounded using the same reasoning as done in the proof of Lemma \ref{lem:epsover31}:
\begin{align*}
\mathbb{P}\left(\left|\int f \mbox{ } d\mu_{X_n}-\int f \mbox{ } d\mu_{\widehat{X}_n}\right|\geq \frac{\epsilon}{3}\right)
&\leq\frac{9 ||f||_{\mathrm{Lip}}^2}{\epsilon^2 n}\sum_{1\leq i,j \leq n}\mathbb{E}\left(([X_n]_{ij}-[\widehat{X}_n]_{ij})^2\right)\\
&=\frac{9 ||f||_{\mathrm{Lip}}^2}{\epsilon^2 n}\cdot n^2\mathbb{E}\left(\frac{1}{n}(Y_{11}-\widehat{Y}_{11})^2\right)\\
&<\eta.
\end{align*}

The second term on the right-hand side of (\ref{eq:otherepsover3}) goes to 0 as $n \rightarrow\infty$ by the case we just proved.

The third term is bounded as done in Lemma \ref{lem:epsover33}:
\begin{align*}
\left|\mathbb{E}\left(\int f \mbox{ } d\mu_{\widehat{X}_n}\right)-\mathbb{E}\left(\int f \mbox{ } d\mu_{X_n}\right)\right|
&\leq \frac{||f||_{\mathrm{Lip}}}{\sqrt{n}}\left(\mathbb{E}\left(\mathrm{Tr}[(X_n-\widehat{X}_n)^2]\right)\right)^{1/2}\\
&= \frac{||f||_{\mathrm{Lip}}}{\sqrt{n}}\left(n^2\mathbb{E}\left(\frac{1}{n}(Y_{11}-\widehat{Y_{11}})^2\right)\right)^{1/2}\\
&<\frac{\epsilon}{3},
\end{align*}
so $\mathbb{P}\left(\left|\mathbb{E}\left(\int f \mbox{ } d\mu_{\widehat{X}_n}\right)-\mathbb{E}\left(\int f \mbox{ } d\mu_{X_n}\right)\right|\geq \frac{\epsilon}{3}\right)=0.$ So
$$
\limsup_{n\rightarrow\infty}\mbox{ }\mathbb{P}\left(\left|\int f \mbox{ } d\mu_{X_n}-\mathbb{E}\left(\int f \mbox{ } d\mu_{X_n}\right)\right|\geq \epsilon\right)<\eta.
$$
Since $\eta>0$ was arbitrary, we have $\mathbb{P}\left(\left|\int f \mbox{ } d\mu_{X_n}-\mathbb{E}\left(\int f \mbox{ } d\mu_{X_n}\right)\right|\geq \epsilon\right)\rightarrow0$ as $n\rightarrow\infty$, giving convergence in probability.
\end{proof}

\section{Growth of the log-Sobolev constant for convolved measures}

Although Theorem \ref{thm:z} gives that $\mu*\gamma_\delta$ satisfies a LSI with some constant $c(\delta)$, the proof
does not give any estimate for what $c(\delta)$ actually is. If $\mu$ does not satisfy a LSI, then $c(\delta)$ should blow up
as $\delta\rightarrow 0$. One would hope for a good upper bound for $c(\delta)$, in terms of $\delta$, that blows up slowly
 as $\delta\rightarrow 0$. However, we show that $c(\delta)$ can, in many cases, blow up very badly as $\delta\rightarrow 0$.

\begin{theorem}\label{thm:blowup}
Let $\mu$ be a compactly supported probability measure with disconnected support. Then for some $C>0$ the log Sobolev constant $c(\delta)$ for
$\mu*\gamma_\delta$ satisfies $c(\delta)\geq \exp\left(C/\delta\right)$ for all sufficiently small $\delta>0$.
\end{theorem}

Before we prove Theorem \ref{thm:blowup}, we briefly demonstrate that such measures described above cannot satisfy a LSI. Suppose
$\mathrm{supp}(\mu)\subseteq [a,b]\cup[c,d]$, with $\mu([a,b])\neq 0,$  $\mu([c,d])\neq 0,$ and $a\leq b<c\leq d$. Let $g$ be a smooth function
such that $g=0$ on $[a,b]$, $g=1$ on $[c,d]$, and $g'=0$ on $[a,b]\cup[c,d]$. Then we easily
compute that $\mathrm{Ent}_\mu (g^2)=\mu([c,d])\cdot \log(1/\mu([c,d]))>0$ but $\int(g')^2 d\mu=0$, so the LSI (\ref{eq:lsi}) cannot hold for any $c$.

We also remark that, in light of Theorem \ref{thm:blowup}, we cannot hope to use the Borel-Cantelli Lemma to improve the convergence in probability stated in
Theorem \ref{thm:evalconverge} to almost sure convergence via this method.
\begin{proof}
Suppose $\mathrm{supp}(\mu)\subseteq [a,b]\cup[c,d]$, with $\mu([a,b])\neq 0,$  $\mu([c,d])\neq 0,$ and $a\leq b<c\leq d$.
Let $p$ be the density of $\mu*\gamma_\delta$. We will show that $D_0$, as defined in Theorem \ref{thm:bg},
is asymptotically greater than or equal to $\exp\left(C/\delta\right)$ for some $C$ as $\delta\rightarrow 0$.

We can assume without loss of generality that the median $m$ of $\mu*\gamma_\delta$ satisfies $m\geq \frac{b+c}{2}$ (if not, we consider
$D_1$ instead of $D_0$). Then
\begin{align*}
D_0
&=\sup_{x<m}\left(\int_{-\infty}^xp(t)dt\cdot\log\frac{1}{\int_{-\infty}^xp(t)dt}\cdot\int_x^m\frac{1}{p(t)}dt\right)\\
&\geq \left[\int_{-\infty}^xp(t)dt\cdot\log\frac{1}{\int_{-\infty}^xp(t)dt}\cdot\int_x^m\frac{1}{p(t)}dt\right]_{x=b}\\
&\geq \int_{-\infty}^bp(t)dt\cdot\log\frac{1}{\int_{-\infty}^bp(t)dt}\cdot\int_b^{(b+c)/2}\frac{1}{p(t)}dt.
\end{align*}
We now bound each of $\int_{-\infty}^bp(t)dt$, $\log(1/\int_{-\infty}^bp(t)dt)$, and $\int_b^{(b+c)/2}(1/p(t))dt$ from below.

First, we claim that
$$
\int_{-\infty}^bp(t)dt\geq \frac{1}{2}\mu([a,b]).
$$
To prove this, we use the definition of $p$ and Fubini's theorem:
\begin{align*}
\int_{-\infty}^bp(t)dt
&=\int_{-\infty}^b\int_a^d\frac{1}{\sqrt{2\pi\delta}}\exp\left(-\frac{(t-s)^2}{2\delta}\right)d\mu(s)\hspace{1mm}dt\\
&=\int_a^d\int_{-\infty}^b\frac{1}{\sqrt{2\pi\delta}}\exp\left(-\frac{(t-s)^2}{2\delta}\right)dt\hspace{1mm}d\mu(s)\\
&\geq\int_a^b\int_{-\infty}^{b-s}\frac{1}{\sqrt{2\pi\delta}}\exp\left(-\frac{t^2}{2\delta}\right)dt\hspace{1mm}d\mu(s)\\
&\geq\int_a^b\int_{-\infty}^0\frac{1}{\sqrt{2\pi\delta}}\exp\left(-\frac{t^2}{2\delta}\right)dt\hspace{1mm}d\mu(s)\\
&=\int_a^b\frac{1}{2}d\mu(s)\\
&=\frac{1}{2}\mu([a,b]).
\end{align*}

To bound $\log(1/\int_{-\infty}^bp(t)dt)$ from below, we bound $\int_{-\infty}^bp(t)dt$ from above:
\begin{align*}
\int_{-\infty}^bp(t)dt
&=\int_{-\infty}^b\int_a^d\frac{1}{\sqrt{2\pi\delta}}\exp\left(-\frac{(t-s)^2}{2\delta}\right)d\mu(s)\hspace{1mm}dt\\
&=\int_a^d\int_{-\infty}^b\frac{1}{\sqrt{2\pi\delta}}\exp\left(-\frac{(t-s)^2}{2\delta}\right)dt\hspace{1mm}d\mu(s)\\
&=\int_a^b\int_{-\infty}^b\frac{1}{\sqrt{2\pi\delta}}\exp\left(-\frac{(t-s)^2}{2\delta}\right)dt\hspace{1mm}d\mu(s)
+\int_c^d\int_{-\infty}^b\frac{1}{\sqrt{2\pi\delta}}\exp\left(-\frac{(t-s)^2}{2\delta}\right)dt\hspace{1mm}d\mu(s)\\
&\leq\int_a^b 1\hspace{1mm}d\mu(s)
+\int_c^d\int_{-\infty}^{b-s}\frac{1}{\sqrt{2\pi\delta}}\exp\left(-\frac{t^2}{2\delta}\right)dt\hspace{1mm}d\mu(s)\\
&\leq\mu([a,b])+\int_c^d\int_{-\infty}^0\frac{1}{\sqrt{2\pi\delta}}\exp\left(-\frac{t^2}{2\delta}\right)dt\hspace{1mm}d\mu(s)\\
&=\mu([a,b])+\frac{1}{2}\mu([c,d])\\
&=1-\frac{1}{2}\mu([c,d]).
\end{align*}

To bound $\int_b^{(b+c)/2}(1/p(t))dt$, first note that if $b\leq t\leq\frac{b+c}{2}$, then for $s\leq b$ we have $t-s\geq t-b\geq 0$ so that
$\exp(-(t-s)^2)\leq\exp(-(t-b)^2)$; also, for $s\geq c$ we have $s-t\geq c-t\geq t-b\geq 0$ so that $\exp(-(t-s)^2)\leq\exp(-(t-b)^2)$. Therefore
\begin{align*}
p(t)&=\int_a^d\frac{1}{\sqrt{2\pi\delta}}\exp\left(-\frac{(t-s)^2}{2\delta}\right)d\mu(s)\\
&=\int_a^b\frac{1}{\sqrt{2\pi\delta}}\exp\left(-\frac{(t-s)^2}{2\delta}\right)d\mu(s)
+\int_c^d\frac{1}{\sqrt{2\pi\delta}}\exp\left(-\frac{(t-s)^2}{2\delta}\right)d\mu(s)\\
&\leq\int_a^b\frac{1}{\sqrt{2\pi\delta}}\exp\left(-\frac{(t-b)^2}{2\delta}\right)d\mu(s)
+\int_c^d\frac{1}{\sqrt{2\pi\delta}}\exp\left(-\frac{(t-b)^2}{2\delta}\right)d\mu(s)\\
&=\int_a^d\frac{1}{\sqrt{2\pi\delta}}\exp\left(-\frac{(t-b)^2}{2\delta}\right)d\mu(s)\\
&=\frac{1}{\sqrt{2\pi\delta}}\exp\left(-\frac{(t-b)^2}{2\delta}\right),
\end{align*}
so that
\begin{align*}
\int_b^{(b+c)/2}\frac{1}{p(t)}dt
&\geq \int_b^{(b+c)/2}\sqrt{2\pi\delta}\cdot\exp\left(\frac{(t-b)^2}{2\delta}\right)dt\\
&=\sqrt{2\pi}\cdot\delta\int_0^{(c-b)/(2\delta)}\exp\left(\frac{u^2}{2}\right)du,
\hspace{5mm}\mbox{where}\hspace{5mm}u=\frac{t-b}{\sqrt{\delta}}\\
&\sim\frac{2\sqrt{2\pi}}{c-b}\cdot\delta^{3/2}\cdot\exp\left(\frac{(c-b)^2}{8\delta}\right)
\end{align*}
as $\delta\rightarrow 0$.

Putting these estimates together, we get
\begin{align*}
D_0
&\geq \int_{-\infty}^bp(t)dt\cdot\log\frac{1}{\int_{-\infty}^bp(t)dt}\cdot\int_b^{(b+c)/2}\frac{1}{p(t)}dt\\
&\geq \frac{1}{2}\mu([a,b])\cdot\log\frac{1}{1-\frac{1}{2}\mu([c,d])}\cdot\frac{2\sqrt{2\pi}}{c-b}\cdot\delta^{3/2}\cdot\exp\left(\frac{(c-b)^2}{8\delta}\right)\\
&\geq\exp\left(\frac{C}{\delta}\right)\hspace{3mm}\mbox{as}\hspace{3mm}\delta\rightarrow 0
\end{align*}
for any $C<(c-b)^2/8$.
\end{proof}

\section{Gaussian convolutions in higher dimensions}\label{sec:higherdim}

We are interested in an analogue of Theorem \ref{thm:z} for a compactly supported probability measure $\mu$ on $\mathbb{R}^n,$ $n\geq2$. If $\mu$ is the
product of probability measures on $\mathbb{R}$, then for any $\delta>0$, $\mu*\gamma_{\delta}$ satisfies a LSI by Theorem \ref{thm:z} and the following product property of LSIs (see \cite{Gr75}, p. 1074, Rk. 3.3):
\begin{theorem}\label{thm:gross}
$\mathrm{(Gross).}$ Let $\nu_1,\nu_2$ be probability measures on $\mathbb{R}^{n_1}$ and $\mathbb{R}^{n_2}$ that each satisfy a LSI with constants $c_1, c_2$,
respectively. Then the probability measure $\nu_1\otimes\nu_2$ on $\mathbb{R}^{n_1+n_2}$ satisfies a LSI with constant $\max(c_1, c_2)$.
\end{theorem}

We remark that Theorem \ref{thm:gross} is stated above in a more general context than what was considered in \cite{Gr75}; for a detailed proof of
Theorem \ref{thm:gross}, see \cite{Gu09}, p.52.\\

Another sufficient condition for which a measure $\nu$ satisfies a LSI is the following (see \cite{BE85}, p. 199, Cor. 2):

\begin{theorem}\label{thm:bakem}
$\mathrm{(Bakry\mbox{, }\acute{E}mery, Ledoux).}$ Let $\nu$ be a probability measure on $\mathbb{R}^n$ with smooth, strictly positive density $p$. If there exists
$c>0$ such that $\mathrm{Hess}(-\log p)(x)-\frac{1}{c}I_n$ is positive semidefinite for all $x\in\mathbb{R}^n$, where $I_n$ is the $n\times n$ identity matrix, then $\nu$ satisfies a LSI with constant $c$.
\end{theorem}

We remark that the above theorem was stated by Bakry and \'Emery in \cite{BE85} in a slightly different context from what is given here; Theorem \ref{thm:bakem} was stated in the above form by Ledoux; for a proof of Theorem \ref{thm:bakem}, see \cite{Gu09}, p.55.\\

Using Theorem \ref{thm:bakem}, we show below that $\mu*\gamma_\delta$ satisfies a LSI for sufficiently {\it large} $\delta$.

\begin{theorem}
Let $\mu$ be a probability measure on $\mathbb{R}^n$ whose support is contained in a ball of radius $R$. Then $\mu*\gamma_\delta$ satisfies a LSI for all
$\delta>2R^2n$.
\end{theorem}

\begin{proof}
Suppose $\delta>2R^2n$. Since satisfaction of a LSI is clearly unchanged by translation, we may suppose that the ball containing $\mathrm{supp}(\mu)$ is centered at 0. Then
$\mu*\gamma_\delta$ has smooth, strictly positive density $p$ given by
$$
p(x)=\int (2\pi\delta)^{-n/2}\exp\left(\frac{-(x-y)^2}{2\delta}\right)d\mu(y)=\int d\nu_x(y),
$$
where $d\nu_x(y)=(2\pi\delta)^{-n/2}\exp\left(\frac{-(x-y)^2}{2\delta}\right)d\mu(y)$. It is then straightforward to compute that, for $i\neq j$,
\begin{align*}
&\partial_i p(x)=-\frac{1}{\delta}\left(x_i\int d\nu_x(y)-\int y_i\hspace{1mm}d\nu_x(y)\right),\\
&\partial_{ii} p(x)=-\frac{1}{\delta^2}\left(\delta\int d\nu_x(y)-x_i^2\int d\nu_x(y)+2x_i\int y_i\hspace{1mm}d\nu_x(y)
-\int y_i^2\hspace{1mm}d\nu_x(y)\right),\hspace{5mm}\mbox{and}\\
&\partial_{ij} p(x)=\frac{1}{\delta^2}\left(x_i x_j\int d\nu_x(y)-x_i\int y_j\hspace{1mm}d\nu_x(y)-x_j\int y_i\hspace{1mm}d\nu_x(y)
+\int y_i y_j\hspace{1mm}d\nu_x(y)\right);
\end{align*}
differentiation under the integral is justified by the Dominated Convergence Theorem since the integrands are smooth and have bounded partial derivatives
of all orders.

We now show $\delta\cdot\mathrm{Hess}(-\log p)$ converges uniformly to the $n\times n$ identity matrix as $\delta\rightarrow\infty$. For $i\neq j$,
\begin{align*}
(\partial_i p\cdot\partial_j p-p\cdot\partial_{ij} p)(x)
=&\frac{1}{\delta^2}\left(x_i\int d\nu_x(y)-\int y_i\hspace{1mm}d\nu_x(y)\right)\left(x_j\int d\nu_x(y)-\int y_j\hspace{1mm}d\nu_x(y)\right)\\
-&\frac{1}{\delta^2}\int d\nu_x(y) \left(x_i x_j\int d\nu_x(y)-x_i\int y_j\hspace{1mm}d\nu_x(y)-x_j\int y_i\hspace{1mm}d\nu_x(y)
+\int y_i y_j\hspace{1mm}d\nu_x(y)\right)\\
=&\frac{1}{\delta^2}\left(\int y_i\hspace{1mm}d\nu_x(y)\int y_j\hspace{1mm}d\nu_x(y)-\int y_i y_j\hspace{1mm}d\nu_x(y)\right),
\end{align*}
so
\begin{align*}
\partial_{ij}(-\log p(x))&=\frac{\partial_i p(x)\partial_j p(x)-p(x)\partial_{ij} p(x)}{p(x)^2}\\
&=\frac{\int y_i\hspace{1mm}d\nu_x(y)\int y_j\hspace{1mm}d\nu_x(y)-\int y_i y_j\hspace{1mm}d\nu_x(y)}{\delta^2\left(\int d\nu_x(y)\right)^2}.
\end{align*}
Thus
\begin{align*}
\left|\delta\cdot\partial_{ij}(-\log p(x))\right|&\leq\frac{\int |y_i|\hspace{1mm}d\nu_x(y)\int |y_j|\hspace{1mm}d\nu_x(y)+\int |y_i| |y_j|\hspace{1mm}d\nu_x(y)}
{\delta\left(\int d\nu_x(y)\right)^2}\\
&\leq\frac{R^2\left(\int d\nu_x(y)\right)^2+R^2\left(\int d\nu_x(y)\right)^2}{\delta\left(\int d\nu_x(y)\right)^2}\\
&=\frac{2R^2}{\delta}.
\end{align*}
We also compute
\begin{align*}
\partial_{ii}(-\log p(x))&=\frac{(\partial_i p(x))^2-p(x)\partial_{ii} p(x)}{p(x)^2}\\
&=\frac{\left(\int y_i\hspace{1mm}d\nu_x(y)\right)^2+\delta\left(\int d\nu_x(y)\right)^2-\int d\nu_x(y)\int y_i^2\hspace{1mm}d\nu_x(y)}
{\delta^2\left(\int d\nu_x(y)\right)^2},
\end{align*}
so
\begin{align*}
\left|\delta\cdot\partial_{ii}(-\log p(x))-1\right|
&=\left|\frac{\left(\int y_i\hspace{1mm}d\nu_x(y)\right)^2-\int d\nu_x(y)\int y_i^2\hspace{1mm}d\nu_x(y)}{\delta\left(\int d\nu_x(y)\right)^2}\right|\\
&\leq\frac{\left(\int |y_i|\hspace{1mm}d\nu_x(y)\right)^2+\int d\nu_x(y)\int y_i^2\hspace{1mm}d\nu_x(y)}{\delta\left(\int d\nu_x(y)\right)^2}\\
&\leq\frac{R^2\left(\int d\nu_x(y)\right)^2+R^2\left(\int d\nu_x(y)\right)^2}{\delta\left(\int d\nu_x(y)\right)^2}\\
&=\frac{2R^2}{\delta}.
\end{align*}

So $\delta\cdot\mathrm{Hess}(-\log p)=I_n+A(\delta)$, where $A(\delta)$ is an $n\times n$ real symmetric matrix whose entries are all uniformly bounded in
absolute value by $2R^2/\delta$. We therefore have for all ${\bf v}\in\mathbb{R}^n$, $c\in\mathbb{R}$,
\begin{align*}
\langle(\mathrm{Hess}(-\log p)-\frac{1}{c}I_n){\bf v},{\bf v}\rangle
&=\langle\frac{1}{\delta}(I_n+A(\delta)){\bf v},{\bf v}\rangle-\frac{1}{c}||{\bf v}||^2\\
&=\left(\frac{1}{\delta}-\frac{1}{c}\right)||{\bf v}||^2+\frac{1}{\delta}\langle A(\delta){\bf v},{\bf v}\rangle\\
&\geq \left(\frac{1}{\delta}-\frac{1}{c}\right)||{\bf v}||^2-\frac{1}{\delta}|\langle A(\delta){\bf v},{\bf v}\rangle|
\end{align*}
But by Cauchy-Schwarz, we have
\begin{align*}
|\langle A(\delta){\bf v},{\bf v}\rangle|^2
&=|\sum_{i,j}A_{ij}v_iv_j|^2\\
&\leq\sum_{i,j}|A_{ij}|^2\cdot\sum_{i,j}|v_iv_j|^2\\
&\leq\sum_{i,j}\left(\frac{2R^2}{\delta}\right)^2\cdot\sum_{i}|v_i|^2\cdot\sum_{j}|v_j|^2\\
&=n^2\left(\frac{2R^2}{\delta}\right)^2||{\bf v}||^2\cdot||{\bf v}||^2\\
&=\left(\frac{2R^2n}{\delta}||{\bf v}||^2\right)^2,
\end{align*}
so for sufficiently large $c$,
\begin{align*}
\langle(\mathrm{Hess}(-\log p)-\frac{1}{c}I_n){\bf v},{\bf v}\rangle
&\geq  \left(\frac{1}{\delta}-\frac{1}{c}\right)||{\bf v}||^2-\frac{1}{\delta}|\langle A(\delta){\bf v},{\bf v}\rangle|\\
&\geq  \left(\frac{1}{\delta}-\frac{1}{c}\right)||{\bf v}||^2-\frac{1}{\delta}\cdot\frac{2R^2n}{\delta}||{\bf v}||^2\\
&= \frac{1}{\delta^2}\left(\delta-2R^2n-\frac{\delta}{c}\right)||{\bf v}||^2\\
&\geq 0
\end{align*}
since $\delta>2R^2n$. So by Theorem \ref{thm:bakem}, $\mu*\gamma_\delta$ satisfies a LSI.
\end{proof}

We are ultimately interested in showing that  $\mu*\gamma_\delta$ satisfies a LSI for all $\delta>0$, not just for large $\delta$. This will require a different approach
than the one used in this paper, as there is no known analogue of Theorem \ref{thm:bg} for higher dimensions, and is a topic currently being explored by the author.

\section{Acknowledgements}
The author would like to thank his Ph.D. advisor, Todd Kemp, both for posing this problem and for his valuable insight and discussions regarding this problem.

\end{document}